\documentclass[a4paper,11pt,reqno]{amsart}
\pdfoutput=1

\usepackage{amsmath,amssymb,amsthm,mathrsfs}
\usepackage{mathtools}
\usepackage[all,cmtip]{xy}
\usepackage{bm}
\usepackage[a4paper,margin=1in]{geometry}
\usepackage{lmodern}
\usepackage[T1]{fontenc}

\usepackage{comment}

\usepackage{bm}
\usepackage{microtype}

\usepackage[UKenglish]{isodate}

\usepackage[shortlabels]{enumitem}

\theoremstyle{plain}

\newtheorem{theorem}{Theorem}
\newtheorem{lemma}[theorem]{Lemma}
\newtheorem{corollary}[theorem]{Corollary}

\theoremstyle{definition} 
\newtheorem{definition}[theorem]{Definition}

\theoremstyle{remark}
\newtheorem{remark}[theorem]{Remark}

\makeatletter
\let\c@equation\c@theorem
\makeatother
\numberwithin{equation}{section}

\usepackage[hidelinks]{hyperref}

\newcommand{\lra}{\longrightarrow}
\newcommand{\Hom}{\mathrm{Hom}}
\newcommand{\ho}{\operatorname{ho}}
\newcommand{\cd}[2][]{\vcenter{\hbox{\xymatrix#1{#2}}}}
\newcommand{\dtwocell}[3][0.5]{\ar@{}[#2] \ar@{=>}?(#1)+/u  0.2cm/;?(#1)+/d 0.2cm/^{#3}}
\newcommand{\fatpullbackcorner}[1][dr]{\save*!/#1-1.75pc/#1:(-1,1)@^{|-}\restore}
\newcommand{\fatpushoutcorner}[1][dr]{\save*!/#1+1.75pc/#1:(1,-1)@^{|-}\restore}

\title{A counterexample in quasi-category theory}
\author{Alexander Campbell}
\thanks{The author gratefully acknowledges the support of Australian Research Council Discovery Grant DP160101519.}
\address{Centre of Australian Category Theory \\ Macquarie University \\ NSW 2109 \\ Australia}
\email{alexander.campbell@mq.edu.au}
\urladdr{http://web.science.mq.edu.au/~alexc/}

\date{10 April 2019}

\begin{document}

\begin{abstract}
We give an example of a morphism of simplicial sets which is a monomorphism, bijective on $0$-simplices, and a weak categorical equivalence, but which is not inner anodyne. This answers an open question of Joyal. Furthermore, we use this morphism to refute a plausible description of the class of fibrations in Joyal's model structure for quasi-categories.
\end{abstract}

\maketitle

\section{Introduction} \label{secintro}
A morphism of simplicial sets is said to be \emph{inner anodyne} (or an \emph{inner anodyne extension}) if it belongs to the saturated class generated by the inner horn inclusions $\Lambda^n_k \lra \Delta^n$, $0<k<n$, i.e.\ if it can be expressed as a retract of a countable composite of pushouts of coproducts of inner horn inclusions. It is easily shown that every inner anodyne extension is (i) a monomorphism, (ii) bijective on $0$-simplices, and (iii) a weak categorical equivalence (i.e.\ a weak equivalence in Joyal's model structure for quasi-categories \cite{joyalbarcelona}, \cite{MR2522659}). This statement admits a partial converse: if a morphism of simplicial sets whose codomain is a quasi-category satisfies the properties (i)--(iii), then it is inner anodyne. 

The purpose of this note is to show that the complete converse of this statement is false; we give in \S\ref{secex} an example of a morphism of simplicial sets (therein denoted $f \colon \Delta^1 \lra S$, see Definition \ref{def})  which satisfies the properties (i)--(iii), but which is not inner anodyne  (see Theorem \ref{thm}). This answers an open question of Joyal (see \cite[\S 2.10]{joyalnotes}). 

The existence of such a morphism yields answers to a few related open questions. For example, we show (Corollary \ref{cor1}) that the class of inner anodyne extensions does not have the left cancellation property.\footnote{Stevenson \cite{MR3812459} proved that the class of inner anodyne extensions does have the right cancellation property.} Furthermore, we refute a plausible description of the fibrations in Joyal's model structure for quasi-categories. Recall that a morphism of simplicial sets is said to be an \emph{inner fibration} if it has the right lifting property with respect to the inner horn inclusions (and hence all inner anodyne extensions). Joyal proved that a morphism of simplicial sets $p \colon X \lra Y$ whose codomain $Y$ is a quasi-category is a fibration in the model structure for quasi-categories if and only if it is an inner fibration and an isofibration on homotopy categories (i.e.\ the induced functor $\ho(p) \colon \ho(X) \lra \ho(Y)$ between homotopy categories is an isofibration). We show (Corollary \ref{cor2}) that these two properties fail to describe the fibrations with arbitrary codomain.

\section{The counterexample} \label{secex}
\begin{definition} \label{def}
Let $S$ denote the simplicial set freely generated by a $2$-simplex $\alpha$ whose 2nd face $d_2(\alpha)$ is degenerate, which is defined by the following pushout diagram in the category of simplicial sets.
\begin{equation*}
\cd[@=3em]{
\Delta^1 \ar[r] \ar[d]_-{d^2}  & \Delta^0 \ar[d]^-x \\
\Delta^2 \ar[r]_-{\alpha} & \fatpushoutcorner S
}
\end{equation*}
Let $f \colon \Delta^1 \lra S$ denote the composite of the morphisms $d^1 \colon \Delta^1 \lra \Delta^2$ and $\alpha \colon \Delta^2 \lra S$, which picks out the 1st face of the $2$-simplex $\alpha$ in $S$.
\end{definition}

The simplicial set $S$ may be displayed as a quotient of the $2$-simplex $\Delta^2$ as below. 
\begin{equation*}
\cd[@=3em]{
0 \ar[rr]^-{02} \ar[dr]_-{01} & {}  \dtwocell[0.4]{d}{012} & 2 \\
& 1 \ar[ur]_-{12} & \\
}
\qquad
\underset{\alpha}{\longmapsto}
\qquad
\cd[@=3em]{
x \ar[rr]^-f \ar@{=}[dr] & {}  \dtwocell[0.4]{d}{\alpha} & y \\
& x \ar[ur]_-g & \\
}
\end{equation*}
Thus $S$ has two $0$-simplices $x$ and $y$ (say), two non-degenerate $1$-simplices $f$ and $g$ with $d_1(f) = d_1(g) = x$ and $d_0(f) = d_0(g) = y$, and one non-degenerate $2$-simplex $\alpha$ with $d_2(\alpha) = s_0(x)$, $d_1(\alpha) = f$, and $d_0(\alpha) = g$. 

\begin{remark}
The simplicial set $S$ of Definition \ref{def} can also be described as the ``right suspension'' of the simplicial set $\Delta^1$, which is denoted variously as $C_R(\Delta^1)$ in \cite[\S4.3]{MR2764043}, $\Delta^2_{1|2}$ in \cite[\S15.4]{MR3221774}, and $\Sigma^r\Delta^1$ in \cite[\S7.1]{rvcomp}. Thus a morphism $u$ from  $S$ to any simplicial set $X$ amounts to a $1$-simplex in the right hom-space $\Hom_X^R(u(x),u(y))$ of $X$ \cite[\S1.2.2]{MR2522659}. Moreover, the morphism $f \colon \Delta^1 \lra S$ can be described as the right suspension of the morphism $d^1 \colon \Delta^0 \lra \Delta^1$.
\end{remark}

We shall see from the following list of its properties that the morphism $f \colon \Delta^1 \lra S$ of Definition \ref{def} is a counterexample to the statements considered in \S \ref{secintro}.

\begin{lemma} \label{lem}
The morphism $f \colon \Delta^1 \lra S$ is
\begin{enumerate}[font=\normalfont, label=(\roman*)]
\item a monomorphism,
\item bijective on $0$-simplices,
\item a weak categorical equivalence, and
\item an inner fibration.
\end{enumerate}
\end{lemma}
\begin{proof}
The properties (i) and (ii) are immediate from the definition of $f$. 

To prove (iii), 
observe that the morphism $g \colon \Delta^1 \lra S$ (the composite of $d^0 \colon \Delta^1 \lra \Delta^2$ and $\alpha \colon \Delta^2 \lra S$) is a weak categorical equivalence; for it is the pushout of the inner horn inclusion $\Lambda^2_1 \lra \Delta^2$ along the morphism $\Lambda^2_1 \lra \Delta^1$ which picks out the $(2,1)$-horn 
\begin{equation*}
\cd{
0 \ar@{=}[r]^-{00} & 0 \ar[r]^-{01} & 1
}
\end{equation*}
in $\Delta^1$, as displayed on the left below.
Since the two morphisms $f, g \colon \Delta^1 \lra S$ admit a common retraction, namely the unique morphism $r \colon S \lra \Delta^1$ that sends $\alpha$ to the degenerate $2$-simplex $001$ in $\Delta^1$, it follows by the two-out-of-three property that $f$ is a weak categorical equivalence.

\begin{equation*}
\cd[@=3em]{
\Lambda^2_1 \ar[r]^-{(01,-, 00)} \ar[d] & \Delta^1 \ar[d]^-g \\
\Delta^2 \ar[r]_-{\alpha} & S \fatpushoutcorner
}
\qquad
\qquad
\cd[@=3em]{
\Lambda^n_k \ar@{..>}[r]_-{\exists!} \ar[d] \ar@/^1.25pc/[rr] & \Lambda^2_0 \ar[r] \ar[d] \fatpullbackcorner & \Delta^1 \ar[d]^-f \\
\Delta^n \ar@/_1.25pc/[rr]_-{} \ar@{..>}[r]^-{\exists} \ar@{..>}[ur]|-{\exists} & 
\Delta^2 \ar@{->>}[r]^-{\alpha} & S
}
\end{equation*}

To prove (iv), observe that the pullback of $f \colon \Delta^1 \lra S$ along the epimorphism $\alpha \colon \Delta^2 \lra S$ is the outer horn inclusion $\Lambda^2_0 \lra \Delta^2$, which is an inner fibration (since it is a morphism between nerves of categories). It then follows by a simple argument (indicated by the diagram on the right above) that $f \colon \Delta^1 \lra S$ is an inner fibration.
\end{proof}

\begin{theorem} \label{thm}
There exists a morphism of simplicial sets which is a monomorphism, bijective on $0$-simplices, and a weak categorical equivalence, but which is not inner anodyne.
\end{theorem}
\begin{proof}
By Lemma \ref{lem}, the morphism $f \colon \Delta^1 \lra S$ of Definition \ref{def} has the first three listed properties. To prove that $f$ is not inner anodyne, it suffices to observe that $f$ is an inner fibration by Lemma \ref{lem} and is evidently not an isomorphism; for if a morphism is both inner anodyne and an inner fibration, then it has the right lifting property with respect to itself, and is therefore an isomorphism.
\end{proof}

\begin{corollary} \label{cor1}
The class of inner anodyne extensions does not have the left cancellation property; that is, there exists a composable pair of monomorphisms of simplicial sets $i,j$ such that $j$ and $ji$ are inner anodyne, but such that $i$ is not inner anodyne.
\end{corollary}
\begin{proof}
Let $i \colon A \lra B$ be a morphism as described in Theorem \ref{thm}, and let $j \colon B \lra C$ be an inner anodyne extension with $C$ a quasi-category (as  may be constructed by the small object argument). Then their composite $ji \colon A \lra C$ is a monomorphism, bijective on $0$-simplices, and a weak categorical equivalence, and hence is inner anodyne, since its codomain is a quasi-category (see for instance \cite[Lemma 2.19]{stevbifib}). But $i$ is not inner anodyne.
\end{proof}

\begin{corollary} \label{cor2}
There exists a morphism of simplicial sets which is an inner fibration and an isofibration on homotopy categories, but which is not a fibration in Joyal's model structure for quasi-categories.
\end{corollary}
\begin{proof}
This can be proved as a direct corollary of Theorem \ref{thm} by a factorisation and retract argument, but we will show instead that the morphism $f \colon \Delta^1 \lra S$ of Definition \ref{def} is itself such a morphism as in the statement. By Lemma \ref{lem}, $f$ is an inner fibration, and is an isomorphism (and hence an isofibration) on homotopy categories, since it is bijective on $0$-simplices and a weak categorical equivalence. But $f$ is a trivial cofibration in the model structure for quasi-categries, since it is both a monomorphism and a weak categorical equivalence, and hence is not a fibration, since it is not an isomorphism. 
\end{proof}

\begin{remark}
In unpublished work of Joyal, there is constructed a model structure on the category of simplicial sets whose cofibrations are the ``immersions'' (i.e.\ those morphisms sent to monomorphisms by the reflection to the full subcategory of reduced simplicial sets), whose weak equivalences are the bijective-on-$0$-simplices weak categorical equivalences, and whose fibrant objects are the quasi-categories. Theorem \ref{thm} implies moreover that, while every fibration in this model structure is an inner fibration (and conversely for morphisms to quasi-categories), not every inner fibration is a fibration.
\end{remark}

\end{document}